\documentclass[reqno]{amsart}
\usepackage{amsthm,amsmath,amssymb}
\usepackage{stmaryrd,mathrsfs}
\usepackage{esint}
\usepackage{tikz}
\usepackage{hyperref}
\usepackage{enumerate}

\allowdisplaybreaks

\newcommand{\ud}[0]{\,\mathrm{d}}

\newcommand{\abs}[1]{|#1|}

\newcommand{\Norm}[2]{\|#1\|_{#2}}

\newcommand{\BNorm}[2]{\Big\|#1\Big\|_{#2}}

\newcommand{\ave}[1]{\langle #1\rangle}

\newcommand{\bddlin}[0]{\mathscr{L}}

\newcommand{\BMO}[0]{\operatorname{BMO}}
\newcommand{\supp}[0]{\operatorname{supp}}
\newcommand{\loc}[0]{\operatorname{loc}}

\newcommand{\R}{\mathbb{R}}
\newcommand{\C}{\mathbb{C}}
\newcommand{\N}{\mathbb{N}}

\newcommand{\eps}[0]{\varepsilon}

\swapnumbers \numberwithin{equation}{section}

\theoremstyle{plain}
\newtheorem{theorem}[equation]{Theorem}
\newtheorem{proposition}[equation]{Proposition}
\newtheorem{corollary}[equation]{Corollary}
\newtheorem{lemma}[equation]{Lemma}

\theoremstyle{definition}
\newtheorem{definition}[equation]{Definition}

\theoremstyle{remark}
\newtheorem{remark}[equation]{Remark}

\makeatletter
\@namedef{subjclassname@2010}{%
  \textup{2010} Mathematics Subject Classification}
\makeatother

\begin{document}

\title[Schatten characterisations of commutators]{$A_\infty$-invariance of oscillatory norms, and\\ Schatten characterisations of commutators}


\author{Tuomas Hyt\"onen}
\address{Aalto University,
Department of Mathematics and Systems Analysis,
P.O. Box 11100, FI-00076 Aalto, Finland}
\email{tuomas.hytonen@aalto.fi}

\thanks{The author was supported by the Research Council of Finland, projects 364208 and 371637.}

\keywords{Commutator, Schatten class, Bessel operator, $A_\infty$ condition}
\subjclass[2020]{42B20, 42B35, 46E36, 47B10, 47B47}




\begin{abstract}
Schatten class properties of commutators $[b,T]$ of pointwise multipliers $b$ and singular integral operators $T$
have been characterised in a variety of settings. An abstract framework, covering many of these results as special cases, was proposed by the author [arXiv:2411.02613]. However, recent results about commutators of the concrete Bessel--Riesz transforms by Fan--Li--Sukochev--Zanin [arXiv:2411.14928] are beyond this abstract setting.

In this work, we present an extension of the framework of [arXiv:2411.02613], introducing two measures $\mu$ and $\nu$ that are $A_\infty$-equivalent to each other. The commutators act on a given space $L^2(\mu)$, but the characterising function space norms of the multiplier $b$ are taken with respect to another measure $\nu$. In this way, assumptions like Ahlfors regularity and Poincar\'e inequality on the original measure $\mu$ may be relaxed, as long as there is an $A_\infty$-equivalent measure $\nu$ that satisfies these assumptions. In the Bessel example, the original $\mu$ fails to be Ahlfors regular, but $\nu$ is simply the Lebesgue measure.

Within this framework, the Schatten norm characterisations of commutators of the Bessel--Riesz transforms at the critical-index by Fan--Li--Sukochev--Zanin [op cit.]\ are recovered by a completely different argument, replacing non-commutative techniques by real-variable harmonic analysis and hardly using any specifics of the Bessel setting. As a by-product, we also obtain a simpler characterisation in the non-critical case, replacing an {\em ad hoc} Besov space of Fan--Lacey--Li--Xiong [J. Funct. Anal. 2026] by a classical Besov space.
\end{abstract}

\maketitle

\tableofcontents

\section{Introduction}

Different mapping properties of commutators
\begin{equation*}
  [b,T]:f\mapsto bT(f)-T(bf)
\end{equation*}
of pointwise multipliers $b$ and singular integrals $T$ play a role in various areas of mathematics, like factorisation of function spaces \cite{CRW:76} and compensated compactness in nonlinear PDE \cite{CLMS}. Specifically, questions of quantitative compactness of $[b,T]$, measured in terms of Schatten--Lorentz $S^{p,q}$ norms (Definition \ref{def:Spq}), arise in connection with the quantised calculus of Connes \cite[Chapter 4]{Connes:book} and are particularly relevant for the critical parameters $(p,q)=(d,\infty)$, where $d$ is an underlying dimension \cite{CST,Frank,FSZ:23,LMSZ}.

Over the past few years, classical characterisations \cite{JW:82,RS:NWO} of commutator Schatten properties have been extended to various new settings like weighted Euclidean spaces \cite{GLW:23},
Heisenberg groups \cite{FLL:23,FLMSZ}, more general Carnot groups \cite{LXY},
and the so-called Bessel setting \cite{FLLX}. Besides its intrinsic interest, the last-mentioned case serves as a useful model for a general theory, featuring a concrete example of a doubling metric measure space with different upper and lower dimensions, thus failing to be Ahlfors regular (see Definitions \ref{def:metric} and \ref{def:Ahlfors}).

A unified framework covering many of these previous results as special cases was provided in \cite{Hyt:Sp}; however, within three weeks (in terms of the original arXiv dates in November 2024), it was challenged by \cite{FLSZ} with a characterisation for the concrete Bessel--Riesz transforms in the important critical-index case, which was not covered by the abstract assumptions of \cite{Hyt:Sp}. The goal of this paper is to place these bounds of \cite{FLSZ} into a further general framework, a natural extension of \cite{Hyt:Sp}. This leads to a new proof of these bounds, quite transparent in our opinion, and offers a new explanation of the underlying phenomenon in a broader context. At the same time, the scope of the bounds of \cite{FLSZ} is significantly expanded from the concrete Bessel setting to a wide class of operators and spaces.

We should note that, besides the uniform bounds of \cite[Theorem 1.2]{FLSZ}, the same paper also proves related spectral asymptotics in \cite[Theorem 1.3]{FLSZ}, and these latter remain beyond the scope of our approach. However, to obtain these asymptotics for all relevant functions rather than just test functions, the uniform bounds of \cite[Theorem 1.2]{FLSZ} are used in the proof of the asymptotics in \cite[Theorem 1.3]{FLSZ}. Our present results, extending \cite[Theorem 1.2]{FLSZ}, not only offer an alternative (simpler, in our opinion) approach to this major component of the proof of \cite[Theorem 1.3]{FLSZ}, but they should similarly serve as a step towards the proof of spectral asymptotics in new situations. In prospective new cases where spectral asymptotics for commutators are hoped for, the specifics of the situation may still be needed for the eventual asymptotic considerations, while the uniform non-asymptotic bounds will be available from the general theory presented here.

To discuss our contributions in more detail, we recall that the following norm equivalence (first due to \cite{RS:NWO} in the Euclidean context, and recently extended in \cite{Hyt:Sp}) holds in great generality:
\begin{equation}\label{eq:Spq-Opq}
  \Norm{[b,T]}{S^{p,q}(L^2(\mu))}\sim\Norm{b}{\operatorname{Osc}^{p,q}(\mu)},
\end{equation}
where $S^{p,q}(L^2(\mu))$ is the usual Schatten--Lorentz space of operators on $L^2(\mu)$ (Definition \ref{def:Spq}) and $\operatorname{Osc}^{p,q}(\mu)$ is an oscillatory norm space recalled in Definition \ref{def:Osc-pq}. Here and throughout the paper, we use the following convention. Whenever $X$ is a (quasi-)normed space, we write $\Norm{f}{X}:=\infty$ for $f\notin X$. Using this convention, we see that the concise equivalence ``$\Norm{f}{X}\sim\Norm{g}{Y}$'' as in \eqref{eq:Spq-Opq} encodes the longer statement ``$f\in X$ if and only if $g\in Y$, and in this case $\Norm{f}{X}\sim\Norm{g}{Y}$''.

While the equivalence \eqref{eq:Spq-Opq} is very general, the identification of the abstract oscillation norm on the right with more familiar, and maybe more useful, function space norms is only available under more restrictive assumptions on the underlying space and the measure $\mu$. Notably, in the important critical-index case of the parameters $(p,q)$, an identification of the oscillation norm with a suitable Sobolev-type norm (see Definition \ref{def:Hajlasz}),
\begin{equation*}
  \Norm{b}{\operatorname{Osc}^{d,\infty}(\mu)}\sim\Norm{b}{\dot M^{1,d}(\mu)},
\end{equation*}
was established in \cite{Hyt:Sp,HK:W1p} for Ahlfors $d$-regular measures $\mu$ (Definition \ref{def:Ahlfors}). This Ahlfors regularity fails for the Bessel measure $\ud m_\lambda=x_{n+1}^{2\lambda}\ud x$ on $\R^{n+1}_+:=\R^n\times(0,\infty)$, yet \cite{FLSZ} were still able to prove (for the Bessel--Riesz transforms $R_{\lambda,j}$ in place of $T$) that
\begin{equation*}
  \Norm{[b,T]}{S^{d,\infty}(L^2(\R_+^{n+1},m_\lambda))}\sim\Norm{b}{\dot W^{1,d}(\R_+^{n+1})},
  \qquad d=n+1.
\end{equation*}
Here, the Sobolev space $\dot W^{1,d}(\R_+^{n+1})$ on the right is the classical one with respect to the unweighted Lebesgue measure, not with respect to the Bessel measure $m_\lambda$ appearing on the left. The Lebesgue measure on the half-space $\R_+^{n+1}$ is evidently Ahlfors regular; hence, there is still some hidden Ahlfors-regularity in the setting of \cite{FLSZ}, even if not in a form that would allow a direct application of \cite{Hyt:Sp,HK:W1p}.

In this paper, we provide an additional ingredient that is needed to bring the Schatten bounds of \cite{FLSZ} under the scope of \cite{Hyt:Sp,HK:W1p} by means of the following invariance of the oscillation norms under $A_\infty$ changes of the measure:

\begin{proposition}\label{prop:Osc(mu)=Osc(nu)}
Let $\mu,\nu$ be two doubling measures on a metric space $(X,\rho)$ such that they satisfy the $A_\infty$ relation
\begin{equation}\label{eq:Ainfty}
  \mu(E)\leq\eps\mu(B)\quad\Rightarrow\quad\nu(E)\leq(1-\delta)\nu(B)
\end{equation}
for some $\delta,\eps\in(0,1)$, whenever $E\subset B$ is a measurable subset of a metric ball $B$.

Then for all $p\in(0,\infty)$ and $q\in(0,\infty]$, the oscillatory norms of Definition \ref{def:Osc-pq} satisfy
\begin{equation*}
   \Norm{b}{\operatorname{Osc}^{p,q}(\mu)}
   \sim\Norm{b}{\operatorname{Osc}^{p,q}(\nu)}
\end{equation*}
for all measurable functions $b$, with implied constants independent of $b$.
\end{proposition}

This can be thought of as an extension of a classical $A_\infty$-invariance result of \cite[Theorem 5]{MW:75} about $\BMO(\mu)$, which formally corresponds to $\operatorname{Osc}^{\infty,\infty}(\mu)$. Proposition \ref{prop:Osc(mu)=Osc(nu)} is based on similar ideas, but seems to be new, to our knowledge. But perhaps the main novelty here is using Proposition \ref{prop:Osc(mu)=Osc(nu)} in combination with \eqref{eq:Spq-Opq} and the characterisations of $\operatorname{Osc}^{p,q}(\nu)$ from \cite{Hyt:Sp,HK:W1p} to obtain the following extension of \cite[Theorem 1.1]{Hyt:Sp}, where we characterise Schatten norms of the commutator $[b,T]$ on $L^2(\mu)$ in terms of function space norms with another measure~$\nu$. The usefulness of this change of measure becomes apparent in the application to the Bessel setting further below.

\begin{theorem}\label{thm:Sp-nu}
Let $d\in[1,\infty)$, let $(X,\rho)$ be a metric space, and suppose that
\begin{enumerate}[\rm(i)]
  \item\label{it:mu-nu} $\mu$ and $\nu$ are doubling measures on $(X,\rho)$ that satisfy the $A_\infty$ relation \eqref{eq:Ainfty};
  \item\label{it:sep-dim} every ball $B(x,R)$ in $X$ contains at most $C(R/r)^d$ points $x_i$ with $\rho(x_i,x_j)\geq r$, uniformly in $x\in X$ and $0<r\leq R<\infty$; and
  \item\label{it:CZ} $T\in\bddlin(L^2(\mu))$ is a singular integral operator
\begin{equation*}
  Tf(x)=\int_X K(x,y)f(y)\ud\mu(y),\qquad x\notin\supp f,
\end{equation*}
 with non-degenerate standard kernel $K$ of H\"older exponent $\eta>(1-\frac d2)_+$ (Definition \ref{def:CZK}), where $x_+:=\max\{x,0\}$.
\end{enumerate}
Then the following conclusions are valid for all functions $b\in L^1_{\loc}(\mu)$:
\begin{enumerate}[\rm(1)]\setcounter{enumi}{-1}
  \item\label{it:p>d} If $p\in(d,\infty)$, then
\begin{equation}\label{eq:p>d}
     \Norm{[b,T]}{S^p(L^2(\mu))}
     \sim
     \Norm{b}{\dot B^p(\nu)}:=\Big(\iint_{X\times X}\frac{\abs{b(x)-b(y)}^p}{\nu(B(x,\rho(x,y)))^2}
  \ud\nu(x)\ud\nu(y)\Big)^{\frac 1p}.
\end{equation}

    \item\label{it:p>dA} If $p\in(d,\infty)$ and $(X,\rho,\nu)$ is Ahlfors $d$-regular, then
\begin{equation*}
     \Norm{[b,T]}{S^p(L^2(\mu))}
    \sim\Norm{b}{\dot B^{\frac dp}_{p,p}(\nu)}:=\Big(\iint_{X\times X}\frac{\abs{b(x)-b(y)}^p}{\rho(x,y)^{2d}}
  \ud\nu(x)\ud\nu(y)\Big)^{\frac 1p}.
\end{equation*}

  \item\label{it:p<d} If $d>1$ and $(X,\rho,\nu)$ has lower dimension $d$ and supports the $(1,d)$-Poincar\'e inequality (Definition \ref{def:Poincare}), then
\begin{equation}\label{eq:b=const}
  [b,T]\in S^p(L^2(\mu))\quad\Leftrightarrow\quad b=\text{constant a.e.},\quad
  \text{for all }p\in(0,d].
\end{equation}

  \item\label{it:p=d}
  If $d>1$ and $(X,\rho,\nu)$ is Ahlfors $d$-regular and supports the $(1,d_-)$-Poincar\'e inequality (Definition \ref{def:Poincare}) for some $d_-\in[1,d)$, then
\begin{equation*}
     \Norm{[b,T]}{S^{d,\infty}(L^2(\mu))}
    \sim\Norm{b}{\dot M^{1,d}(\nu)},
\end{equation*}
where $\dot M^{1,d}(\nu)$ is the Haj\l{}asz--Sobolev space (Definition \ref{def:Hajlasz}).

  \item\label{it:A2} If $w\in A_2(\mu)$ (Definition \ref{def:A2}), then all claims above remain valid with the $L^2(\mu)$  space on the left replaced by the weighted space $L^2(w)=L^2(w\ud\mu)$.
\end{enumerate}
\end{theorem}

When $\nu=\mu$, Theorem \ref{thm:Sp-nu} is contained in \cite[Theorem 1.1]{Hyt:Sp}. We make the following remarks about the usefulness of using two different measures:

\begin{remark}
\eqref{it:p>d} and \eqref{it:p>dA}: In general, the commutator characterisation involves the {\em ad hoc} Besov space $\dot B^p$, which only coincides with a classical-style $\dot B^{\frac dp}_{p,p}$ under Ahlfors regularity. By Theorem \ref{thm:Sp-nu}, it is not necessary for the original measure $\mu$ to be Ahlfors-regular, as long as it is $A_\infty$-equivalent to an Ahlfors-regular one.

\eqref{it:p<d}: This is the classical cut-off phenomenon first found by \cite{JW:82} in $\R^d$. It was subsequently extended to various Ahlfors regular cases \cite{FLL:23,LXY}. The first extension beyond Ahlfors regularity, highlighting the role of the lower dimension, is from \cite[Theorem 1.5 and Remark 1.6]{FLLX} in the Bessel setting, and this was placed into the general framework of lower dimension $d$ and $(1,d)$-Poincar\'e inequality in \cite[Theorem 1.1]{Hyt:Sp}. Now, we make a further extension, observing that it is not necessary for the original measure to have these properties, as long as an $A_\infty$-equivalent one does. Note that \eqref{eq:b=const} makes no reference to $\nu$; the new measure only appears in the conditions that guarantee \eqref{eq:b=const}. (Under the $A_\infty$-relation \eqref{eq:Ainfty}, it is easy to see that $\mu$ and $\nu$ are mutually absolutely continuous. Hence the notions of $\mu$-a.e.\ and $\nu$-a.e.\ coincide, and we can drop the reference to the measure, writing simply ``a.e.'' without ambiguity.)

\eqref{it:p=d}: This is the critical-index characterisation, first sketched by \cite{CST} in $\R^d$; see also \cite{Frank,FSZ:23,LMSZ} for elaborations. This, too, was extended to several concrete instances of Ahlfors regular spaces \cite{FLMSZ,LXY}, and eventually to an abstract class of such spaces with the additional assumption of the Poincar\'e inequality in \cite{Hyt:Sp}. In the concrete Bessel setting again, \cite{FLSZ} showed that a result like \eqref{it:p=d} can also hold in certain non-Ahlfors situations. Theorem \ref{thm:Sp-nu} reproves and explains this by showing that it is enough for an $A_\infty$-equivalent measure to satisfy the Ahlfors and Poincar\'e assumptions.

\eqref{it:A2}: While weighted norm inequalities at large have been extensively studied both in Euclidean spaces and beyond, Schatten class estimates on $L^2(w)$ are relatively few. For $(X,\rho)=(\R^d,\abs{x-y})$ with $\mu=\nu=\ud x$ and one of the classical Riesz transforms $R_j=\partial_j(-\Delta)^{-\frac12}$ in place of $T$, this weighted result is due to \cite{GLW:23}. For general spaces $(X,\rho)$ and operators $T$ as in Theorem \ref{thm:Sp-nu}, case $\mu=\nu$ was obtained in \cite{Hyt:Sp}. See also \cite{LLW:2wH,LLWW:2wR} for weighted results in another direction of generality, namely, dealing with Schatten bounds between $L^2(w_1)$ and $L^2(w_2)$ for two different weights $w_i\in A_2(\R^n)$, but with the underlying measure $\mu$ being simply the Lebesgue measure $\ud x$.
\end{remark}

We further illustrate Theorem \ref{thm:Sp-nu} in the concrete Bessel setting studied in \cite{FLLX,FLSZ}:

\begin{corollary}\label{cor:Bessel}
Let $n\geq 0$, let
\begin{equation*}
   T=R_{\lambda,j}=\partial_j(-\Delta_\lambda)^{-\frac12},\qquad
   \Delta_\lambda:=\sum_{j=1}^{n+1}\partial_j^2+\frac{2\lambda}{x_{n+1}}\partial_{n+1},
\end{equation*}
be one the Bessel--Riesz transforms on $\R_+^{n+1}:=\R^n\times(0,\infty)$, and $b\in L^1_{\loc}(\R^{n+1}_+)$.
\begin{enumerate}[\rm(1)]
  \item\label{it:p>n+1} If $p\in(n+1,\infty)$, then 
\begin{equation*}
  \Norm{[b,T]}{S^p(L^2(m_\lambda))}
  \sim \Norm{b}{\dot B^{\frac{n+1}{p}}_{p,p}(\R^{n+1}_+)} 
  :=\Big(\iint_{\R^{n+1}_+\times\R^{n+1}_+}\frac{\abs{b(x)-b(y)}^p}{\abs{x-y}^{2(n+1)}}\ud x\ud y\Big)^{\frac1p}.
\end{equation*}

  \item\label{it:p<n+1}  If $n\geq 1$ and $p\in(0,n+1]$, then
\begin{equation*}
     [b,T]\in S^p(L^2(m_\lambda))\quad\Leftrightarrow\quad
     b=\text{constant a.e.}
\end{equation*}

  \item\label{it:p=n+1} If $n\geq 1$, then 
\begin{equation}\label{eq:p=n+1}
  \Norm{[b,T]}{S^{n+1,\infty}(L^2(m_\lambda))}
  \sim \Norm{b}{\dot W^{1,n+1}(\R^{n+1}_+)} 
  :=\Big(\iint_{\R^{n+1}_+}\abs{\nabla b(x)}^{n+1}\ud x\Big)^{\frac{1}{n+1}}.
\end{equation}

  \item\label{it:Bessel-A2} If $w\in A_2(m_\lambda)$ (Definition \ref{def:A2}), then all claims above remain valid with $L^2(m_\lambda)$ on the left replaced by the weighted space $L^2(w)=L^2(w\ud m_\lambda)$.
\end{enumerate}
\end{corollary}

This is essentially the special case of Theorem \ref{thm:Sp-nu} corresponding to $\ud\mu=\ud m_\lambda=x_{n+1}^{2\lambda}\ud x$ (the Bessel measure) and $\ud\nu=\ud x$ (the Lebesgue measure on $\R_+^{n+1}$).

We note that \eqref{it:p=n+1} is a restatement of \cite[Theorem 1.2]{FLSZ}; however, we reprove it here in a completely different way. Likewise, \eqref{it:p<n+1} was proved in \cite[Theorem 1.5]{FLLX} under the {\em a priori} assumption that $b\in C^2(\R_+^{n+1})$ and in \cite[Example 1.23]{Hyt:Sp} as stated. Nevertheless, Theorem \ref{thm:Sp-nu} offers a slight variant of the argument compared to \cite{Hyt:Sp}: it is no longer necessary to know about a Poincar\'e inequality in $(\R_+^{n+1},\abs{x-y},m_\lambda)$ (although it was verified in \cite[Proposition 4.2]{Hyt:Sp}), but simply to use the standard Poincar\'e inequality in $(\R_+^{n+1},\abs{x-y},\ud x)$.

While \eqref{it:p<n+1} and \eqref{it:p=n+1} are simply restatements of known results, \eqref{it:p>n+1} is a new variant of \cite[Theorem 1.4]{FLLX}, where a similar result was obtained using an {\em ad hoc} Besov norm with respect to the measure $m_\lambda$:
\begin{equation*}
  \Norm{[b,T]}{S^p(L^2(m_\lambda))}
  \sim \Norm{b}{\dot B^p(m_\lambda)},
\end{equation*}
with $\dot B^p(m_\lambda)$ defined as in \eqref{eq:p>d} for $\nu=m_\lambda$.
Corollary \ref{cor:Bessel} shows that this exotic Besov norm reduces to a completely classical one. Together with the other cases of Corollary \ref{cor:Bessel}, this means that the Schatten properties of commutators of the Bessel--Riesz transforms are fully characterised in terms classical functions spaces, which is a new result even in this concrete setting.

The weighted extension stated in \eqref{it:Bessel-A2} is a further novelty also in case \eqref{it:p=n+1}, where the unweighted case is a restatement of \cite[Theorem 1.2]{FLSZ}. As the reader can observe in the proof below, this extension comes almost for free in our approach, while it seems non-obvious whether such an extension can be incorporated into the argument of \cite{FLSZ}.

The rest of the paper is organised as follows. In Section \ref{sec:basic}, we collect some general definitions that were already used in this Introduction. In Section \ref{sec:Ainfty}, we introduce the oscillatory norms appearing in Proposition \ref{prop:Osc(mu)=Osc(nu)} and, after some preparations, give the proof of the said proposition. Theorem \ref{thm:Sp-nu} and Corollary \ref{cor:Bessel} are then proved in the following Sections \ref{sec:Sp-nu} and \ref{sec:Bessel}, respectively. We conclude with some final comments comparing our approach and that of \cite{FLSZ} in Section \ref{sec:comments}.

\section{Basic definitions}\label{sec:basic}

For the convenience of reference, we collect here some basic definitions that were already used in the Introduction.
While these definitions are needed to understand the statements of some of our results, we note that most of them will be used only indirectly in this paper, in the sense that these notions appear as input or output of some results that we quote from elsewhere, but hardly in any hands-on computations.

\begin{definition}\label{def:Spq}
Let $H$ is a Hilbert space and $A,F\in\bddlin(H)$ be a bounded linear operators on $H$.
We write $\operatorname{rank}F\leq n$ if there are elements $e_k,f_k\in H$ such that
\begin{equation*}
  F(h)=\sum_{k=1}^n (h,e_k)f_k
\end{equation*}
for all $h\in H$. The $n$th approximation number of $A$ is defined by
\begin{equation*}
  a_n(A):=\inf\{\Norm{A-F}{}:F\in\bddlin(H),\operatorname{rank}F\leq n\},
\end{equation*}
where $n\in\N:=\{0,1,2,\ldots\}$. For $p\in(0,\infty)$ and $q\in(0,\infty]$, the Schatten--Lorentz $S^{p,q}(H)$ (quasi-)norm and space are defined by
\begin{equation*}
\begin{split}
  \Norm{A}{S^{p,q}(H)} &:=\Norm{(a_n(A))_{n=0}^\infty}{\ell^{p,q}}, \\
   S^{p,q}(H) &:=\{A\in\bddlin(H):\Norm{A}{S^{p,q}(H)}<\infty\}.
\end{split}
\end{equation*}
We abbreviate $S^p(H):=S^{p,p}(H)$.
\end{definition}

\begin{definition}\label{def:metric}
A metric measure space is a triplet $(X,\rho,\mu)$, where $X$ is a set, $\rho$ is a metric on $X$, and $\mu$ is a Borel measure on $X$. The space is called doubling if
\begin{equation*}
  \mu(B(x,2R))\lesssim\mu(B(x,R))
\end{equation*}
uniformly in $x\in X$ and $R>0$.
\end{definition}

\begin{definition}\label{def:CZK}
A function $K:\{(x,y)\in X\times X: x\neq y\}\to\C$ is called a standard kernel of H\"older exponent $\eta\in(0,1]$ on $(X,\rho,\mu)$ if
\begin{equation*}
  \abs{K(x,y)}\lesssim\frac{1}{\mu(B(x,\rho(x,y))}
\end{equation*}
and, whenever $\rho(x,x')\ll\rho(x,y)$,
\begin{equation*}
  \abs{K(x,y)-K(x',y)}
  +\abs{K(y,x)-K(y,x')}
  \lesssim\frac{1}{\mu(B(x,\rho(x,y))}\Big(\frac{\rho(x,x')}{\rho(x,y)}\Big)^\eta.
\end{equation*}
We say that such kernel is non-degenerate if for every $x\in X$ and $r(0,\infty)$, there is a point $y\in X$ with $\rho(x,y)\sim r$ such that
\begin{equation*}
  \abs{K(x,y)}+\abs{K(y,x)}\gtrsim\frac{1}{\mu(B(x,\rho(x,y))}.
\end{equation*}
\end{definition}

\begin{definition}\label{def:Ahlfors}
For $0\leq d\leq D<\infty$, we say that a metric measure space $(X,\rho,\mu)$ has lower dimension $d$ and upper dimension $D$ if
\begin{equation*}
  \Big(\frac{R}{r}\Big)^d\lesssim\frac{\mu(B(x,R))}{\mu(B(x,r))}\lesssim\Big(\frac{R}{r}\Big)^D
\end{equation*}
uniformly for all $x\in X$ and all $0<r\leq R\lesssim\operatorname{diam}(X)$.
The space is called Ahlfors $d$-regular, if it has both upper and lower dimension $d$.
\end{definition}

\begin{definition}\label{def:Poincare}
We say that a space $(X,\rho,\mu)$ satisfies the $(1,p)$-Poincar\'e inequality if there are constant $\lambda,c_P>0$ such that, for every Lipschitz function $f:X\to\R$ and every ball $B(x,R)$ in $X$,
\begin{equation*}
  \fint_{B(x,R)}\abs{f-\ave{f}_{B(x,R)}}\ud\mu
  \leq c_P\Big(\fint_{B(x,\lambda R)}\operatorname{lip}f(x)^p \ud\mu\Big)^{\frac1p},
\end{equation*}
where
\begin{equation*}
  \operatorname{lip}f(x):=\liminf_{t\to 0}\sup_{\rho(x,y)\leq t}\frac{\abs{f(x)-f(y)}}{t}.
\end{equation*}
\end{definition}

\begin{definition}\label{def:Hajlasz}
For a measurable function $f$ on $(X,\rho,\mu)$, we say that a measurable $h:X\to[0,\infty]$ is a Haj\l{}asz upper gradient of $f$ if
\begin{equation*}
  \abs{f(x)-f(y)}\leq\rho(x,y)(h(x)+h(y))\qquad\text{for $\mu$-a.e. }x,y\in X.
\end{equation*}
For $p\in(0,\infty]$, the Haj\l{}asz--Sobolev norm of $f$ is
\begin{equation*}
  \Norm{f}{\dot M^{1,p}(\mu)}
  :=\inf\Big\{\Norm{h}{L^p(\mu)}: h\text{ is a Haj\l{}asz upper gradient of }f\Big\}.
\end{equation*}
\end{definition}

\begin{definition}\label{def:A2}
On a doubling metric measure space $(X,\rho,\mu)$, a measurable function $w:X\to[0,\infty]$ is called an $A_2(\mu)$ weights, denoted by $w\in A_2(\mu)$ if
\begin{equation*}
  [w]_{A_2(\mu)}:=\sup_B\Big(\fint_B w\ud\mu\Big)\Big(\fint_B\frac{1}{w}\ud\mu\Big)
\end{equation*}
is finite, where the supremum is taken over all metric balls $B$ in $X$.
\end{definition}

\section{Proof of Proposition \ref{prop:Osc(mu)=Osc(nu)} on the $A_\infty$ invariance}\label{sec:Ainfty}

We begin by introducing the oscillatory spaces $\operatorname{Osc}^{p,q}(\mu)$ featuring in Proposition \ref{prop:Osc(mu)=Osc(nu)} and some variants needed in its proof.

For $r\in(0,\infty)$, a set $E$ of finite positive $\mu$-measure and a measurable function $f$, let
\begin{equation*}
  \operatorname{osc}_{r,\mu}(f;E):=\inf_c\Big(\fint_E\abs{f-c}^r\ud\mu\Big)^{\frac 1r}.
\end{equation*}
Note that this is well defined (but possibly $\infty$) for all measurable $f$. It is finite if and only if $1_E f\in L^r(\mu)$. If $r\in[1,\infty)$, and $1_E f\in L^r(\mu)$, then
\begin{equation}\label{eq:Osc-ave}
  \operatorname{osc}_{r,\mu}(f;E)\sim\Big(\fint_E\abs{f-\ave{f}_E^\mu}^r\ud\mu\Big)^{\frac 1r},\quad
  \ave{f}_E^\mu:=\fint_E f\ud\mu.
\end{equation}

We fix a system $\mathscr D$ of dyadic cubes on $(X,\rho)$ in the sense of \cite{HK:12}. For each dyadic cube $Q\in\mathscr D$ with ``centre'' $z_Q$ and ``side-length'' $\ell(Q)$, let $B_Q:=B(z_Q,c\ell(Q))$ be a ``concentric'' ball with a fixed large expansion factor $c$. Then we define:

\begin{definition}\label{def:Osc-pq}
With the notion as above, for all measurable functions $f$ and exponents $p,r\in(0,\infty)$ and $q\in(0,\infty]$, we define
\begin{equation*}
  \Norm{f}{\operatorname{Osc}^{p,q}_r(\mu)}
  :=\BNorm{\Big\{\operatorname{osc}_{r,\mu}(f;B_Q)\Big\}_{Q\in\mathscr D}}{\ell^{p,q}}
\end{equation*}
and
\begin{equation*}
  \Norm{f}{\operatorname{Osc}^{p,q}(\mu)}:=\Norm{f}{\operatorname{Osc}^{p,q}_1(\mu)}.
\end{equation*}
\end{definition}

Using \eqref{eq:Osc-ave}, one sees the equivalence of this definition with a slightly different version appearing in \cite{Hyt:Sp}.

We also consider a finite collection of adjacent dyadic systems $\mathscr D^m$, $m=1,\ldots,M$, with the following additional property; these exist by \cite[Theorem 4.1]{HK:12}:
\begin{equation}\label{eq:adjacent}
  \forall\ B=B(x,t)\ \exists m\in\{1,\ldots,M\},\ \exists\ Q\in\mathscr D^m:\quad
  B\subset Q,\quad\operatorname{diam}(Q)\lesssim r.
\end{equation}
For each of these $\mathscr D^m$, we consider the dyadic oscillation norm
\begin{equation*}
  \Norm{f}{\operatorname{Osc}^{p,q}_r(\mathscr D^m,\mu)}
  :=\BNorm{\Big\{\operatorname{osc}_{r,\mu}(f;Q)\Big\}_{Q\in\mathscr D^m}}{\ell^{p,q}}.
\end{equation*}

\begin{lemma}\label{lem:Osc=Osc(D)}
Let $(X,\rho,\mu)$ be a doubling metric measure space and $p,r\in(0,\infty)$ and $q\in(0,\infty]$.
For all measurable functions $f$ we have the following equivalence with implicit constants independent of $f$:
\begin{equation*}
  \Norm{f}{\operatorname{Osc}^{p,q}_r(\mu)}
  \sim\sum_{m=1}^M \Norm{f}{\operatorname{Osc}^{p,q}_r(\mathscr D^m,\mu)}
\end{equation*}
\end{lemma}

\begin{proof}
$\lesssim$: For every $Q\in \mathscr D$, we apply \eqref{eq:adjacent} to $B_Q$ to find a $Q^m\in\mathscr D^m$ with $B_Q\subset Q^m$ and $\operatorname{diam}(Q^m)\lesssim\ell(Q)$; hence $\mu(Q^m)\lesssim\mu(Q)\sim\mu(B_Q)$.
It follows that $\operatorname{osc}_{r,\mu}(f,B_Q)\lesssim\operatorname{osc}_{r,\mu}(f,Q^m)$. For $t'\neq t$, let $Q^{t'}:=\varnothing$ and $\operatorname{osc}_{r,\mu}(f,\varnothing):=0$.

Under the correspondence $Q\mapsto Q^m=R$, we claim that each $R\in\mathscr D^m$ may have at most boundedly many preimages. Indeed, these $Q$ satisfy $\ell(Q)\sim\ell(R)$, so they belong to boundedly many generations only. Within each fixed generation, these $Q$ are pairwise disjoint, contained in $R$, and their measure $\mu(Q)\gtrsim\mu(R)$ is at least a fixed fraction of $\mu(R)$. Hence, there can be only boundedly many such $Q$ in each generation, and then only boundedly many in the boundedly many generations altogether. It thus follows that
\begin{equation*}
\begin{split}
  \Norm{f}{\operatorname{Osc}^{p,q}_r(\mu)}
  &=\BNorm{\Big\{\operatorname{osc}_{r,\mu}(f;B_Q)\Big\}_{Q\in\mathscr D}}{\ell^{p,q}}
  \lesssim\sum_{m=1}^M \BNorm{\Big\{\operatorname{osc}_{r,\mu}(f;Q^m)\Big\}_{Q\in\mathscr D}}{\ell^{p,q}} \\
  &\lesssim\sum_{m=1}^M \BNorm{\Big\{\operatorname{osc}_{r,\mu}(f;R)\Big\}_{R\in\mathscr D^m}}{\ell^{p,q}}
  =\sum_{m=1}^M \Norm{f}{\operatorname{Osc}^{p,q}_r(\mathscr D^m,\mu)}.
\end{split}
\end{equation*}

$\gtrsim$: It suffices to consider a fixed $m$. For each $R\in\mathscr D^m$, there is $\hat R\in\mathscr D$ of the same side-length such that $z_R\in Q$. Then $R\subset B_{\hat R}$ when expansion factor defining the latter ball is fixed large enough, but also $\mu(B_{\hat R})\lesssim\mu(R)$. Hence $\operatorname{osc}_{r,\mu}(f;R)\lesssim\operatorname{osc}_{r,\mu}(f,B_{\hat R})$.

Reasoning as before, under the correspondence $R\mapsto \hat R=Q$, each $Q\in\mathscr D^m$ may have at most boundedly many preimages. Hence
\begin{equation*}
\begin{split}
  \Norm{f}{\operatorname{Osc}^{p,q}_r(\mathscr D^m,\mu)}
  &=\BNorm{\Big\{\operatorname{osc}_{r,\mu}(f;R)\Big\}_{R\in\mathscr D^m}}{\ell^{p,q}}
  \lesssim \BNorm{\Big\{\operatorname{osc}_{r,\mu}(f;B_{\hat R})\Big\}_{R\in\mathscr D^m}}{\ell^{p,q}} \\
  &\lesssim\BNorm{\Big\{\operatorname{osc}_{r,\mu}(f;B_Q)\Big\}_{Q\in\mathscr D}}{\ell^{p,q}}
  =\Norm{f}{\operatorname{Osc}^{p,q}_r(\mu)}.
\end{split}
\end{equation*}
This completes the proof.
\end{proof}

\begin{lemma}\label{lem:Osc(r)=Osc(s)}
Let $(X,\rho,\mu)$ be a doubling metric measure space and $p,r,s\in(0,\infty)$ and $q\in(0,\infty]$.
For all measurable functions $f$ we have the following equivalence with implicit constants independent of $f$:
\begin{equation*}
  \Norm{f}{\operatorname{Osc}^{p,q}_r(\mathscr D^m,\mu)}
  \sim\Norm{f}{\operatorname{Osc}^{p,q}_s(\mathscr D^m,\mu)}
\end{equation*}
and
\begin{equation*}
  \Norm{f}{\operatorname{Osc}^{p,q}_r(\mu)}
  \sim\Norm{f}{\operatorname{Osc}^{p,q}_s(\mu)}.
\end{equation*}
\end{lemma}

\begin{proof}
The first claim is \cite[Proposition 11.3]{Hyt:Sp}. The second one follows by combining the first claim with Lemma \ref{lem:Osc=Osc(D)}.
\end{proof}

In the Euclidean space, there are many equivalent versions of the $A_\infty$ condition \eqref{eq:Ainfty}. Not all of these equivalences remain valid in greater generality (see \cite{KK:11}), but the following result is enough for our purposes:

\begin{proposition}[\cite{KK:11}, Corollary 4.19]\label{prop:KK}
Let $\mu,\nu$ be doubling measures on a metric space $(X,\rho)$. Then $\mu,\nu$ satisfy the $A_\infty$ relation if and only if $\ud\nu=w\ud\mu$, where the density $w$ satisfies the following reverse H\"older inequality for some $t>1$ and all metric balls $B$:
\begin{equation*}
  \Big(\fint_B w^t\ud\mu\Big)^{\frac1t}\lesssim\fint_B w\ud\mu
  =\frac{\nu(B)}{\mu(B)}.
\end{equation*}
\end{proposition}

Equipped with these tools, we are ready to give:

\begin{proof}[Proof of Proposition \ref{prop:Osc(mu)=Osc(nu)}]
Let $p,r\in(0,\infty)$ and $q\in(0,\infty]$.
Let $B$ be an arbitrary metric ball, and let $w$ and $t>1$ be as in Proposition \ref{prop:KK}. Then
\begin{equation*}
\begin{split}
  \operatorname{osc}_{r,\nu}(f;B)
  &=\inf_c\Big(\frac{1}{\nu(B)}\int_{B}\abs{f-c}^r w\ud\mu\Big)^{\frac1r} \\
  &=\inf_c\Big(\frac{\mu(B)}{\nu(B)}\Big)^{\frac1r}\Big(\fint_{B}\abs{f-c}^r w\ud\mu\Big)^{\frac1r} \\
  &\leq\inf_c\Big(\frac{\mu(B)}{\nu(B)}\Big)^{\frac1r}\Big(\fint_{B}\abs{f-c}^{rt'}\ud\mu\Big)^{\frac{1}{rt'}}
     \Big(\fint_{B} w^t\ud\mu\Big)^{\frac{1}{tr}} \\
  &\lesssim\inf_c\Big(\frac{\mu(B)}{\nu(B)}\Big)^{\frac1r}\Big(\fint_{B}\abs{f-c}^{rt'}\ud\mu\Big)^{\frac{1}{rt'}}
     \Big(\frac{\nu(B)}{\mu(B)}\Big)^{\frac1r}
     \quad\text{by Proposition \ref{prop:KK}} \\
   &=\operatorname{osc}_{rt',\mu}(f;B).
\end{split}
\end{equation*}
Using this with $B=B_Q$ for each $Q\in\mathscr D$ and taking the $\ell^{p,q}$ norm over $\mathscr D$, we obtain
\begin{equation*}
  \Norm{f}{\operatorname{Osc}^{p,q}_r(\nu)}
  \lesssim\Norm{f}{\operatorname{Osc}^{p,q}_{rt'}(\mu)}
  \lesssim\Norm{f}{\operatorname{Osc}^{p,q}_{r}(\mu)}
\end{equation*}
using Lemma \ref{lem:Osc(r)=Osc(s)} with $s=rt'$ in the last step.

By taking complements, it is immediate to see that the assumed $A_\infty$ relation \eqref{eq:Ainfty} is symmetric in $\mu$ and $\nu$. Thus we also obtain
\begin{equation*}
  \Norm{f}{\operatorname{Osc}^{p,q}_r(\mu)}
  \lesssim\Norm{f}{\operatorname{Osc}^{p,q}_{r}(\nu)}.
\end{equation*}
A combination of the two bounds shows that
\begin{equation*}
  \Norm{f}{\operatorname{Osc}^{p,q}_r(\mu)}\sim\Norm{f}{\operatorname{Osc}^{p,q}_r(\nu)},
\end{equation*}
and taking $r=1$ completes the proof.
\end{proof}

\section{Proof of Theorem \ref{thm:Sp-nu} on general commutators}\label{sec:Sp-nu}

The special case $\mu=\nu$ of Theorem \ref{thm:Sp-nu} is contained in \cite[Theorem 1.1]{Hyt:Sp}. The proof of Theorem \ref{thm:Sp-nu} as stated is a combination of Proposition \ref{prop:Osc(mu)=Osc(nu)} and main ingredients of the proof of \cite[Theorem 1.1]{Hyt:Sp} from \cite{Hyt:Sp,HK:W1p}, as we will now detail.

In the notation of \cite[Theorem 1.1]{Hyt:Sp}, assumption \eqref{it:sep-dim} of Theorem \ref{thm:Sp-nu} says that $\Delta=d$. Combined with assumption \eqref{it:CZ} of Theorem \ref{thm:Sp-nu}, it follows that
\begin{equation*}
  \frac{\eta}{\Delta}+\frac 12>\frac{1}{d}(1-\frac{d}{2}) + \frac12=\frac 1d,\qquad
  \Big(\frac{\eta}{\Delta}+\frac 12\Big)^{-1}<d
\end{equation*}
Hence, the index $p(\eta)$ of \cite[(1.9)]{Hyt:Sp} satisfies
\begin{equation}\label{eq:p(eta)<d}
  p(\eta):=\max\Big\{1,\Big(\frac{\eta}{\Delta}+\frac 12\Big)^{-1}\Big\}
  \quad\begin{cases} <d, & \text{if }d>1, \\ =1, & \text{if }d=1.\end{cases}
\end{equation}

For all $p\in(p(\eta),\infty)\subset(1,\infty)$ and $q\in[1,\infty]$, we obtain
\begin{equation}\label{eq:Spq-Osc-pq}
\begin{split}
  \Norm{[b,T]}{S^{p,q}(L^2(\mu))}
  &\sim\Norm{b}{\operatorname{Osc}^{p,q}(\mu)}\quad\text{by \cite[Proposition 1.26]{Hyt:Sp}} \\
  &\sim\Norm{b}{\operatorname{Osc}^{p,q}(\nu)}\quad\text{by Proposition \ref{prop:Osc(mu)=Osc(nu)}}.  \\
  &\sim\Norm{b}{\dot B^p(\nu)}\quad\text{by \cite[Proposition 11.5]{Hyt:Sp}, if }p=q.
\end{split}
\end{equation}

\eqref{it:p>d}: By \eqref{eq:p(eta)<d}, we have $p(\eta)\leq d$ for all $d\in[1,\infty)$.
Hence, we can apply \eqref{eq:Spq-Osc-pq} with $p=q\in(d,\infty)\subseteq(p(\eta),\infty)$ to conclude that
\begin{equation*}
  \Norm{[b,T]}{S^{p}(L^2(\mu))}
  \sim\Norm{b}{\dot B^p(\nu)}.
\end{equation*}

\eqref{it:p>dA}: 
If $\nu$ is Ahlfors $d$-regular, it is immediate that
  $\Norm{b}{\dot B^p(\nu)}\sim\Norm{b}{\dot B^{\frac{d}{p}}_{p,p}(\nu)}.$

\eqref{it:p<d}: The direction ``$\Leftarrow$'' is obvious. Turning to ``$\Rightarrow$'', if $[b,T]\in S^p(L^2(\mu))$ for some $p\in(0,d]$, then $[b,T]\in S^d(L^2(\mu))$ by the standard containments between the $S^p$ spaces. Noting that we now assume $d>1$, it follows from \eqref{eq:p(eta)<d} that $d\in(p(\eta),\infty)$, and hence
\begin{equation*}
  \Norm{[b,T]}{S^d(L^2(\mu)}
  \sim \Norm{b}{\dot B^d(\nu)}
    \quad\text{by \eqref{eq:Spq-Osc-pq} with }p=q=d\in(p(\eta),\infty).
\end{equation*}
Under the assumptions of lower dimension $d$ and $(1,d)$-Poincar\'e inequality made in part \eqref{it:p<d}, \cite[Theorem 1.5(2)]{HK:W1p} says that $\Norm{b}{\dot B^d(\nu)}<\infty$ only if $b$ is constant a.e.

\eqref{it:p=d}: We recall from \cite{HK:W1p} the auxiliary function and measure
\begin{equation*}
  m_b^\nu(x,t):=\fint_{B(x,t)}\abs{b-\ave{b}_{B(x,t)}^\nu}\ud\nu,\quad
  \ud\nu_d(x,t):=t^{-d-1}\ud t\ud\nu(x),
\end{equation*}
where $(x,t)\in X\times(0,\infty)$. Recalling again that $d>1$ and hence $d\in(p(\eta),\infty)$ by \eqref{eq:p(eta)<d}, we obtain
\begin{equation*}
\begin{split}
  \Norm{[b,T]}{S^{d,\infty}(L^2(\mu))}
  &\sim \Norm{b}{\operatorname{Osc}^{d,\infty}(\nu)}
    \quad\text{by \eqref{eq:Spq-Osc-pq} with }(p,q)=(d,\infty) \\
  &\sim \Norm{m_b}{L^{d,\infty}(\nu_d)}
  \quad\text{by}\begin{cases}
      \text{Ahflors $d$-regularity and} \\ \text{\cite[Proposition 1.29(2)]{Hyt:Sp}}
    \end{cases} \\
  &\sim \Norm{b}{\dot M^{1,d}(\nu)}
  \quad\text{by}\begin{cases}
      \text{the $(1,d_-)$-Poincar\'e inequality and} \\ \text{\cite[Theorem 1.1 and Remark 2.9]{HK:W1p}}.
    \end{cases}
\end{split}
\end{equation*}
Concerning the last step, \cite[Theorem 1.1]{HK:W1p} is stated assuming the $(1,d)$-Poincar\'e inequality and completeness of $(X,\rho)$, which imply the $(1,d_-)$-Poincar\'e inequality for some $d_-\in[1,d)$ by \cite{KZ:08}; but it is observed in \cite[Remark 2.9]{HK:W1p} that only this implied condition is actually used in the proof of \cite[Theorem 1.1]{HK:W1p}.

\eqref{it:A2}: In place of the first line of \eqref{eq:Spq-Osc-pq}, we use
\begin{equation}\label{eq:A2-step}
  \Norm{[b,T]}{S^{p,q}(L^2(w))}\sim\Norm{b}{\operatorname{Osc}^{p,q}(\mu)},
\end{equation}
which is also contained in \cite[Proposition 1.26]{Hyt:Sp}. Noting that the right-hand side of \eqref{eq:A2-step} is identical with the right-hand side of the first line of \eqref{eq:Spq-Osc-pq}, the rest of the proof is identical to that already given.

This completes the proof of Theorem \ref{thm:Sp-nu}.\qed

\section{Proof of Corollary \ref{cor:Bessel} in the Bessel setting}\label{sec:Bessel}

The proof consists of simply verifying the assumptions of the general Theorem \ref{thm:Sp-nu} in the special case under consideration. Let $d:=n+1\in[1,\infty)$ and $(X,\rho):=(\R_+^{n+1},\abs{x-y})$.

\eqref{it:mu-nu}: The Lebesgue measure $\ud\nu:=\ud x$ is evidently doubling, and it is easy to check this for the Bessel measure $\ud\mu:=dm_\lambda:=x_{n+1}^{2\lambda}\ud x$ as well. Also, for the power weight $w=x_{n+1}^{2\lambda}$ on $\R_+^{n+1}$, one easily checks that $w\in A_p(\R_+^{n+1})$ as soon as $p>2\lambda+1$. In particular, we have $w\in A_\infty(\R_+^{n+1})$, and one of the many equivalent characterisations of this is \eqref{eq:Ainfty}.

\eqref{it:sep-dim}: It is easy and well known that the Euclidean half-space $(\R_+^{n+1},\abs{x-y})$ satisfies this geometric property with $d=n+1$.

\eqref{it:CZ}: The Bessel--Riesz transforms are standard kernels of (the best possible) H\"older exponent $\eta=1$ on $(X,\rho,\mu)=(\R_+^{n+1},\abs{x-y},m_\lambda)$; this follows from \cite[Lemma 2.5]{FLLX}. The non-degeneracy of these kernels is verified in \cite[Lemma 2.6]{FLLX}. Note that $(1-\frac{d}{2})_+=(1-\frac{n+1}{2})_+\in\{0,\frac12\}$, where both values are strictly less than $\eta=1$, as required.

Since $\ud m_\lambda=x_{n+1}^{2\lambda}\ud x$ and the weight $x_{n+1}^{2\lambda}$ is locally bounded, it follows that $b\in L^1_{\loc}(\R_+^{n+1})\subset L^1_{\loc}(m_\lambda)=L^1_{\loc}(\mu)$. Thus all common assumptions of Theorem \ref{thm:Sp-nu} are verified, and we proceed to the assumptions in the different cases. Note that these case-wise assumptions only involve the easy space $(X,\rho,\nu)=(\R_+^{n+1},\abs{x-y},\ud x)$ with the Lebesgue measure.

Depending on the case, $(\R_+^{n+1},\abs{x-y},\ud x)$ should be either Ahlfors $d$-regular or have lower dimension $d$, which is weaker. This regularity is evident for the Lebesgue measure on $\R_+^{n+1}$ with $d=n+1$. In the last two cases of Theorem \ref{thm:Sp-nu}, we need $d=n+1>1$, which is fine, since we also assume $n\geq 1$ in the corresponding cases of Corollary \ref{cor:Bessel}. Finally, in these cases where $d=n+1>1$, the space $(\R_+^{n+1},\abs{x-y},\ud x)$ should have either the $(1,d)$-Poincar\'e inequality or the stronger $(1,d_-)$-Poincar\'e inequality for some $d_-\in[1,d)$. But it is classical that $(\R_+^{n+1},\abs{x-y},\ud x)$ satisfies even the best possible $(1,1)$-Poincar\'e inequality.

Thus all assumptions of Theorem \ref{thm:Sp-nu} are verified, and hence its conclusions are valid. In cases \eqref{it:p>dA} and \eqref{it:p<d}, it is immediate that they reduce to the corresponding claims \eqref{it:p>n+1} and \eqref{it:p<n+1} of Corollary \ref{cor:Bessel}. In case \eqref{it:p=d}, it suffices to observe that
\begin{equation*}
  \Norm{b}{\dot M^{1,p}(\R_+^{n+1})}\sim
  \Norm{b}{\dot W^{1,p}(\R_+^{n+1})},\qquad p\in(1,\infty],
\end{equation*}
by \cite[Theorem 1]{Hajlasz:96} and take $p=n+1$, recalling that $n+1>1$ under the assumptions of this case. (Strictly speaking, \cite[Theorem 1]{Hajlasz:96} is stated for the full space or a bounded domain in place of the half space $\R_+^{n+1}$, but it is noted immediately below that the theorem ``can be stated in much more general form'', and indeed the case of a half space easily follows along the same lines.)

This completes the proof of Corollary \ref{cor:Bessel}.\qed

\section{Final comments}\label{sec:comments}

The motivation for this investigation came from \cite{FLSZ}, which showed that the abstract framework of \cite{Hyt:Sp} was not broad enough to capture interesting and relevant results in the concrete Bessel setting. By extending the said framework via the use of two different measures, we have managed to view \eqref{eq:p=n+1}, originally obtained as \cite[Theorem 1.2]{FLSZ}, as an instance of general real-variable theory in abstract spaces. Having completed the proof, let us briefly compare our approach with that of \cite{FLSZ}.

In \cite{FLSZ}, two proofs of the upper bound $\lesssim$ in \eqref{eq:p=n+1} are given. The first approach, given in \cite[Sections 3--5]{FLSZ} is based on the following main ingredients:
\begin{enumerate}[\rm(a)]
  \item the corresponding result for the classical Riesz transforms $R_k$ from \cite{LMSZ}; 
  \item explicit formulas relating the kernels of $R_{\lambda,j}$ and $R_k$; and
  \item recent deep results on Schur multipliers from \cite{CGPT} to deduce Schatten bounds for the operators from the obtained pointwise relations of their kernels.
\end{enumerate}
The second approach, in \cite[Sections 6]{FLSZ}, only applies to $n\geq 2$. It depends on:
\begin{enumerate}[\rm(a)]\setcounter{enumi}{3}
  \item an explicit computation of certain $S^2$ (Hilbert--Schmidt) norms by the Fourier--Bessel transform and Fredholm's trace formula;
  \item abstract Cwikel estimates (generalisations of the classical \cite{Cwikel}) from \cite{LSZ:20} to lift the $p=2$ bounds to $p\in(2,\infty)$; and 
  \item explicit computations with the Leibniz rule to pass over some derivatives from the Bessel--Riesz transforms $\partial_j(-\Delta_\lambda)^{-\frac12}$ to the multiplier $f$.
\end{enumerate}

In contrast to these, our approach was built on:
\begin{enumerate}[\rm(a)]\setcounter{enumi}{6}
  \item\label{it:HytSp} a characterisation of Schatten bounds of $[b,T]$ in terms of the $\operatorname{Osc}^{p,q}(\mu)$ norms for general singular integrals $T$ over abstract spaces $(X,\rho,\mu)$, from \cite{Hyt:Sp,RS:NWO};
  \item\label{it:FLLX} verification of the kernel bounds of the general theory for $R_{\lambda,j}$, from \cite{FLLX};
  \item identification of $\operatorname{Osc}^{d,\infty}(\nu)$ with $\dot M^{1,d}(\nu)$ for ``nice'' $\nu$, from \cite{Hyt:Sp,HK:W1p}; and
  \item\label{it:invar} coincidence of $\operatorname{Osc}^{p,q}(\mu)$ and $\operatorname{Osc}^{p,q}(\nu)$ for $A_\infty$-equivalent measures $\mu,\nu$.
\end{enumerate}

In summary, both approaches of \cite{FLSZ} use some deep tools of non-commutative analysis (either the Schur multiplier theorem of \cite{CGPT} or the abstract Cwikel estimates of \cite{LSZ:20}), and the specifics of the Bessel setting are somewhat present throughout the argument. In contrast, our approach is based on real-variable harmonic analysis and the use of the specifics of the Bessel setting is essentially restricted to showing that this setting fits into the general theory in \eqref{it:FLLX}.

Of course, proving any results about the Schatten $S^{p,q}$ spaces, also known as non-commutative $\ell^{p,q}$ spaces, non-commutative analysis cannot be completely avoided, and indeed some results about the $S^{p,q}$ spaces like interpolation and duality are hidden in \eqref{it:HytSp}. Nevertheless, these are arguably lighter (and certainly more classical) aspects of the non-commutative theory than the tools used in \cite{FLSZ}. Besides the new results of this work, providing a common extension of \cite[Theorem 1.1]{Hyt:Sp} and \cite[Theorem 1.2]{FLSZ}, we think that our approach makes the existing \cite[Theorem 1.2]{FLSZ} more approachable to harmonic analysts without a non-commutative background.


\end{document}